\documentclass{amsproc}
\usepackage{amsmath}
\usepackage{amssymb}
\input xy
\xyoption{all}

\theoremstyle{definition}
\newtheorem{theorem}{Theorem}[section]

\theoremstyle{definition}

\newtheorem{example}[theorem]{Example}
\newtheorem{Lemma}[theorem]{Lemma}

\newtheorem{Prop}[theorem]{Proposition}
\theoremstyle{remark}
\newtheorem{remark}[theorem]{Remark}

\numberwithin{equation}{section}
\newcommand{\Cal}[1]{{\mathcal #1}}

\newcommand{\End}{\operatorname{End}}

\newcommand{\Der}{\operatorname{Der}}

\newcommand{\Z}{\mathbb{Z}}
\newcommand{\Q}{\mathbb{Q}}
\newcommand{\R}{\mathbb{R}}
\newcommand{\cmat}{\left(\begin{array}}
\newcommand{\fmat}{\end{array}\right)}

\newcommand{\N}{\mathbb{N}}
\newcommand{\NX}{\mathbb{N}_0[x]^*}
\newcommand{\midN}{\mid_{\N}}
\newcommand{\midZ}{\mid_{\Z}}
\newcommand{\nmidN}{\nmid_{\N}}

\newcommand{\Spec}{\operatorname{Spec}}

\begin{document}

\title[]{Factorizations of polynomials with integral non-negative coefficients}
  
 \author{Federico Campanini}
\address{Dipartimento di Matematica ``Tullio Levi-Civita'', Universit\`a di Padova,\linebreak 35121 Padova, Italy}
\email{federico.campanini@math.unipd.it, facchini@math.unipd.it}
 
\author{Alberto Facchini}
\thanks{Partially supported by Dipartimento di Matematica ``Tullio Levi-Civita'' of Universit\`a di Padova (Project BIRD163492/16 ``Categorical homological methods in the study of algebraic structures'' and Research program DOR1690814 ``Anelli e categorie di moduli'').}



\begin{abstract} We study the structure of the commutative multiplicative monoid $\N_0[x]^*$ of all the non-zero polynomials in $\Z[x]$ with non-negative coefficients. We show that $\N_0[x]^*$ is not a half-factorial monoid and is not a Krull monoid, but has a structure very similar to that of Krull monoids, replacing valuations into $\N_0$ with derivations into $\N_0$. We study ideals, chain of ideals, prime ideals and prime elements of $\N_0[x]^*$. Our monoid $\N_0[x]^*$ is a submonoid of the multiplicative monoid of the ring $\Z[x]$, which is a left module over the Weyl algebra $A_1(\Z)$. 
\end{abstract}

\maketitle

\section{Introduction}

Let $\N_0$ be the set of all non-negative integers $0,1,2,\dots$, let $\N_0[x]$ be the set of all polynomials in the indeterminate $x$ with coefficients in $\N_0$, and $\N_0[x]^*:=\N_0[x]\setminus\{0\}$ be the set of all  non-zero elements of $\N_0[x]$. 

It should not be necessary to present motivations for the study of factorizations of polynomials with integral non-negative coefficients, that is, factorizations in the commutative monoid (semiring) $\N_0[x]$. Here is only one of the infinitely many applications, due to Hashimoto and Nakayama \cite{H, HN}, who showed that the Krull-Schmidt theorem does not hold for finite partially ordered sets. Hashimoto and Nakayama noticed that there are two essentially different factorizations $(x^3+1)(x^2 + x + 1)=(x + 1)(x^4+x^2+1)$ of $x^5+x^4+x^3+x^2+x+1$ into polynomials irreducible in $\N_0[x]$ (atoms of the commutative monoid $\N_0[x]^*$). Now the category of partially ordered sets has coproducts (disjoint unions) and products (direct products with the component-wise order). Let $L=\{0,1\}$ be the partially ordered set with two elements $0<1$.  For every $n\ge 0$, the direct product $L^n$ is a connected partially ordered set with $2^n$ elements and its automorphism group is the symmetric group~$S_n$. Thus we get  two essentially different direct-product decompositions of the partially ordered set $1\dot{\cup}L\dot{\cup}L^2\dot{\cup}L^3\dot{\cup}L^4\dot{\cup}L^5$ into indecomposable partially ordered sets, given by
$(L^3\dot{\cup}1)\times(L^2 \dot{\cup} L \dot{\cup} 1)\cong(L \dot{\cup} 1)\times(L^4\dot{\cup}L^2\dot{\cup}1)$. This technique can be clearly applied in studying the Krull-Schmidt property in any distributive category. 

With these motivations in mind, we study factorizations in the cancellative multiplicative monoid $\NX$, which is known not to be half-factorial (see Section~\ref{Sec2}). Thus we look for further ``regularity properties" of factorizations in $\NX$. Factorizations turn out to be rather regular in Krull monoids. We show that $\NX$ is not a Krull monoid because, for instance, it is not integrally closed (root closed). Nevertheless, $\NX$ has a structure pretty similar to that of Krull monoids. For instance, $\NX$ is a cancellative atomic monoid whose quotient group is a free abelian group (Proposition~\ref{2.2}). Moreover, if instead of monoid morphisms we consider derivations, we get a description of $\NX$ which is very similar to that of Krull monoids (Propositions~\ref{m} and~\ref{m'}). 

In the study of Krull monoids and Krull domains, a predominant role is played by prime ideals of height $1$. This has led us to consider: (1)~Ideals of $\NX$, in particular, chains of principal ideals, which correspond to factorizations in $\NX$ (Sections \ref{id} and~\ref{cpi}). (2)~Prime ideals and prime elements of $\NX$.
(3)~Monoid morphisms of $\NX$ into the additive monoid $\N_0$ (valuations of $\NX$), because for any valuation $v\colon M\to\N_0$ of a commutative monoid $M$, the set of all elements $a\in M$ with $v(a)>0$ is a prime ideal of $M$. 

Since $\NX$ has a description very similar to that of Krull monoids when derivations replace monoid valuations, we consider the Weyl algebra $A_1(\Z)$, over which the ring of polynomials $\Z[x]$ turns out to be a left module  (Section~\ref{n}). The Weyl algebra $A_1$ provides a standard method to study a derivation \cite{FacNazeErice}.

We conclude the paper presenting a further technique to describe factorizations of polynomials with integral non-negative coefficients. Clearly, the set $\NX$ is countable, so that its elements can be parametrized in $\N_0$. We present here a paramentrization of the elements of $\NX$ indexed by the set $\N_0^2$. Our parametrization gives rise to an algorithm that allows to determine factorizations in $\NX$.

\section{Divisiblility in $\N_0[x]^*$}\label{Sec2}

Let $M$ be a commutative cancellative multiplicative monoid and $U(M)$ be the group of its units (=\, invertible elements). If $U(M)=\{1\}$, we will say that $M$ is {\em reduced}. Recall that an element $a \in M$ is an {\em atom} if $a \notin U(M)$ and, whenever $a=xy$ in $M$, either $x \in U(M)$ or $y \in U(M)$. We denote by $\Cal A (M)$ the set of all atoms of $M$.
The monoid $M$ is {\em atomic} if every element of $M\setminus U(M)$ can be written as a product of finitely many atoms, that is, $M$ is generated by the set $\Cal A(M) \cup U(M)$.  An element $a\in M$ is {\em prime}  if it is an atom and, for every $b,c\in M$, $a|bc$ in $M$ implies that either $a|b$ or $a|c$.

In this paper, we assume that {\em our monoids are all commutative, cancellative and atomic}.

 A {\em translation-invariant pre-order}
on a commutative multiplicative monoid $M$ is any relation $\le$ on the set $M$ that is reflexive, transitive and such that $x\le y$ implies $xz\le yz$ for every $x,y,z\in M$.
There is a natural translation-invariant pre-order
on any commutative monoid $M$, called the \index{algebraic pre-order}{\em
algebraic pre-order\/} on $M$, defined, for all $x,y\in M$,  by $x\le y$ if
there exists $z\in M$ such that
$xz=y$. 

\begin{remark}
Notice that the only invertible element in $\NX$ is the identity, and that the algebraic pre-order on $\NX$ is not the restriction of the algebraic pre-order on $\Z[x]$, that is, there exist polynomials $f(x),g(x) \in \NX$ such that $g(x)$ divides $f(x)$ in $\Z[x]$ but not in $\NX$. This implies that, as far as divisibility in $\NX$ is concerned, we must be very careful, because we have two different notions of divisibility in $\NX$. Given $f,g\in \NX$, we write $f\midN g$ if there exists $h\in \NX$ with $g=fh$, and we write $f|_{\Z}g$ if there exists $h\in \Z[x]$ with $g=fh$. Both $\midN$ and $\midZ$ are partial orders on $\NX$ with least element 1 and with no maximal elements. Clearly, $f|_{\N}g$ implies $f|_{\Z}g$, but not conversely, as the factorization $x^3+1=(x+1)(x^2-x+1)$ shows.
\end{remark}

The multiplicative monoid $\N_0[x]^*$ is a cancellative monoid, hence it is contained in its quotient group $G(\N_0[x]^*)$.

\begin{Prop}\label{2.2} The quotient group $G(\N_0[x]^*)$ of the cancellative monoid $\N_0[x]^*$ is a free abelian group.\end{Prop}

\begin{proof} Consider the multiplicative cancellative monoid $\Q[x]^*:=\Q[x]\setminus\{0\}$. Since $\Q[x]$ is a unique factorization domain, we can write every non-zero polynomial in $\Q[x]$ in a unique way as a product of a non-zero rational number and monic irreducible polynomials in $\Q[x]$. Let $I$ be the set of all monic irreducible polynomials in $\Q[x]$. We thus have a monoid isomorphism  $\Q[x]^*\to\Q^*\times\N_0^{(I)}$, where $\N_0^{(I)}$ denotes the factorial monoid (=free commutative monoid) on the free set $I$ of generators. Passing to the field of fractions $\Q(x)$ of $\Q[x]$, we get a group isomorphism $\Q(x)^*\to\Q^*\times\Z^{(I)}$, where now $\Z^{(I)}$ denotes the free abelian group on the free set $I$ of generators. Similarly, every non-zero integer can be written in a unique way as a product of $\pm1$ and prime numbers, and we have a group isomorphism $\Q^*\to\{1,-1\}\times\Z^{(P)}$, where $P$ denotes the sets of prime integers. Hence there is a group isomorphism $\varphi\colon \Q(x)^*\to\{1,-1\}\times\Z^{(P)}\times\Z^{(I)}$, where the number $1$ or $-1$ in $\varphi(f/g)$ is determined by the sign of the quotients of the leading coefficients of $f$ and $g$. Hence $\varphi(\N_0[x]^*)\subseteq \{1\}\times\Z^{(P)}\times\Z^{(I)}$, and thus there is a monoid embedding $\psi\colon \N_0[x]^*\hookrightarrow \Z^{(P)}\times\Z^{(I)}$. Whence $G(\N_0[x]^*)$ is a subgroup of the free abelian group $\Z^{(P)}\times\Z^{(I)}$, and therefore it is a free abelian group.
\end{proof}

Let us briefly recall some basic notions concerning the theory of non-unique factorizations in monoids (see \cite[Chapter 1]{GHK06} for more details).
Let $M$ be a monoid. For every $x \in M \setminus U(M)$, the {\em set of lengths of $x$ in $M$} is defined as the set $L(x):=\{n \mid \mbox{ there are } a_1,\dots,a_n \in \Cal A(M) \mbox{ with }x=a_1 \cdots a_n\}$.
The {\em elasticity of $x$} is the rational number $\rho(x):=\sup L(x)/\min L(x)$.
A monoid $M$ is a BF-{\em monoid} (or a {\em monoid with bounded factorizations}) if $L(x)$ is finite for every $x \in M \setminus U(M)$. If $|L(x)|=1$ for every $x \in M \setminus U(M)$, the monoid is called {\em half-factorial}.

An important tool for investigating the structure of sets of lengths of a monoid $M$ is the {\em set of distances of $M$} (or the $\Delta$-set of $M$). It is the subset $\Delta(M)$ of $\N$ consisting of all $d \in \N$ for which there exist $x \in M \setminus U(M)$ and $\ell \in L(x)$ such that $L(x)\cap[\ell,\ell+d]=\{\ell,\ell+d\}$.
It is immediately seen that a monoid $M$ is half-factorial if and only if $\Delta(M)=\emptyset$.

Of course, the multiplicative monoid $\NX$ is a BF-monoid, but there are several examples demonstrating the non-half-factoriality of $\NX$ (see \cite{vdW12} and \cite{Bru13}). However, in \cite[Section~2]{CCMS09}, the authors show that $\NX$ is surprisingly very far from being half-factorial. In fact, given any rational number $q\geq 1$, there exists an element $f(x) \in \NX$ such that the elasticity $\rho(f)$ of $f(x)$ is equal to $q$ (monoids with this property are said to have {\em full infinite elasticity}) and moreover $\Delta(\NX)=\N$ (see Theorem 2.3 and the last sentence in Section~2 of \cite{CCMS09}).

It is noteworthy that another important result is contained in the proof of \cite[Theorem~2.3]{CCMS09} (and we are grateful to the referee for making us notice this): the set of catenary degrees of $\NX$ equals $\N_{\geq 2}$. Let us explain what it means (again, we refer to \cite[Chapter 1]{GHK06} for more details). Let $M$ be a reduced monoid
and consider the {\em factorization monoid}, which is the free monoid $Z(M)=\Cal F(\Cal A(M))$ whose basis is the set of atoms of $M$. Of course, every element $z \in Z(M)$ has a unique representation of the form $z=\prod_{a \in \Cal A(M)}a^{\nu_a(z)}$, where $\nu_a(z)\in \N_0$ and $\nu_a(z)=0$ for almost all $a \in \Cal A(M)$. For $z \in Z(M)$, the {\em length of $z$} is the natural number given by $|z|:=\sum_{a \in \Cal A(M)}\nu_a(z)$. If $z_1,z_2 \in Z(M)$ and $u=\gcd(z_1,z_2)$, we define the {\em distance} between $z_1$ and $z_2$ as $d(z_1,z_2):=\max\{|u^{-1}z_1|,|u^{-1}z_2|\}$. 
We have a canonical monoid morphism $\pi: Z(M)\rightarrow M$ given by $z\mapsto \prod_{a \in \Cal A(M)}a^{\nu_a(z)}$. For every $x \in M$, the {\em factorizations of $x$} are the elements of the set $Z(x)=\pi^{-1}(x)$. Notice that, for every $x \in M$, we have $L(x)=\{|z|\mid z \in Z(x)\}$.

Let $z,z' \in Z(x)$ be factorizations of $x \in M$ and let $N \in \R_{\geq 0}\cup\{\infty\}$. A finite sequence $z=z_0,z_1,\dots,z_k=z'$ in $Z(x)$ is called an {\em $N$-chain of factorizations} from $z$ to $z'$ if $d(z_{i-1},z_i)\leq N$ for every $i=1,\dots, k$.
Let $c(x) \in \N_0\cup \{\infty\}$ denote the smallest $N \in \N_0\cup \{\infty\}$ such that for any two factorizations $z,z' \in Z(x)$ there exists an $N$-chain of factorizations from $z$ to $z'$. The {\em set of catenary degrees} of $M$ is defined as $\operatorname{Ca}(M):=\{c(a)\mid a \in M, a \mbox{ as at least two factorizations }\} \subseteq \N_{\geq 2}\cup \{\infty\}$.

In the proof of \cite[Theorem~2.3]{CCMS09}, the authors consider the polynomial $g_{n,k}(x)=(x+n)^n(x^2-x+1)(x+k)\in \NX$, where $n,k \in \N$, and they observe that there are only two factorizations of $g_{n,k}(x)$ into atoms of $\NX$, given by
$$
g_{n,k}(x)=[(x+n)^n(x^2-x-1)]\cdot[x+1]^k
$$
and
$$
g_{n,k}(x)=[x+n]^n\cdot[(x^2-x+1)(x+1)]\cdot [x+1]^{k-1},
$$
which have lengths $1+k$ and $n+k$ respectively. Thus $c(g_{n,k}(x))=n$ and from the arbitrarity of $n$, it follows that the set of catenary degrees of $\NX$ equals $\N_{\geq 2}$.

\newpage

In the next proposition, we will determine the prime elements in the monoid $\NX$. We need a lemma.

\begin{Lemma}\label{lemma3}
Let $f_1(x),\dots,f_m(x)\in \NX\setminus \{x\}$ be atoms of degree at least~1. Then there exists $z(x) \in \Z[x]\setminus \NX$ irreducible in $\Z[x]$ such that $\deg(z)\geq 2\max_i\{\deg(f_i)\}$ and $f_i(x)z(x) \in \NX$ for every $i=1,\dots,m$. In particular, if $f(x)\neq x$ is an atom of $\NX$ with $\operatorname{deg}	(f)\geq 1$, then there exist $h(x) \in \NX$ of degree at least $1$ and $z(x) \in \Z[x]\setminus \NX$ such that $f(x)\nmidN h(x)$ and both $f(x)z(x)$ and $h(x)z(x)$ belong to $\NX$.
\end{Lemma}

\begin{proof}
Let $n:=\max_i\{\deg(f_i)\}$. Fix two prime natural numbers $p$ and $q$ such that $p\neq q$ and $q\geq \max_i\{f_i(1)\}$. Notice that $q$ is greater than every coefficient of $f_i(x)$ for every $i=1,\dots, m$. Define a polynomial $z(x)$ in the following way: $z(x):=\sum_{k=0}^{2n}a_k x^k$, where $a_{2n}=q(p+1)$, $a_n=-p$ and $a_k=pq$ for every $k\neq n, 2n$. Then $z(x)$ is a polynomial in $\Z[x]\setminus \NX$ that is irreducible by Eisenstein's criterion. A direct computation shows that $f_i(x)z(x) \in \NX$ for every $i=1,\dots,m$. (As a matter of fact, if $f(x):=\sum_{j=0}^k b_j x^j$ is one of the polynomials $f_1(x),\dots,f_m(x)$, then the coefficients of $f(x)z(x)$ are sums of non-negative integers except for at most one negative term of the form $-b_j p$ for some $j=0,\dots,k$. Moreover, for every coefficient, at least one summand is $pq b_0, (p+1)q b_k$ or $pq b_k$. It follows that all the coefficients of $f(x)z(x)$ are non-negative, that is, $f(x)z(x)\in \NX$).
The second part of the statement is now clear.
\end{proof}


\begin{Prop}\label{Fede} Let $f(x)\neq x$ be an atom of $\NX$ with $\operatorname{deg}	(f)\geq 1$. Then $f(x)$ is not a prime element of $\NX$. In particular, the only prime elements in $\NX$ are the prime integers and the polynomial $x$.
\end{Prop}

\begin{proof} 
By Lemma \ref{lemma3}, there exist $h(x) \in \NX$ and $z(x) \in \Z[x]\setminus \NX$ such that $f(x)\nmidN h(x)$ and both $g(x):=f(x)z(x)$ and $h(x)z(x)$ belong to $\NX$. In particular, $g(x)h(x) \in \NX$ and  $f(x)\mid_\N g(x)h(x)$ but $f(x)\nmid_\N g(x)$ and $f(x)\nmid_\N h(x)$. It follows that $f(x)$ is not a prime element of $\NX$.
\end{proof}

\begin{remark}\label{atomprime}
Let $g(x)$ be an atom of $\NX$ which is not a prime element of $\NX$  and let $z(x) \in \Z[x]\setminus \NX$ be an irreducible element of $\Z[X]$ such that $g(x)z(x)\in \NX$. Then, for every prime element $f(x)$ of $\NX$, $f(x)\nmid_\N g(x)z(x)$. In fact, if $f(x)\mid_\N g(x)z(x)$ then $f(x)\mid_\Z g(x)$, because $z(x)$ is irreducible in $\Z[X]$. But $f(x)$ is a prime integer or the polynomial $x$ (Proposition \ref{Fede}), so $f(x)\mid_\Z g(x)$ if and only if $f(x)\mid_\N g(x)$. Since $g(x)$ is an atom of $\NX$ we must have $f(x)=g(x)$, a contradiction.
\end{remark}

\section{Ideals in $\NX$}\label{id}

Recall that an {\em ideal} of a commutative monoid $M$ is a subset $I$ of $M$ such that  $xy \in I$ for every $x \in I$ and every $y \in M$. A {\em prime ideal}  is a proper subset $P$ of $M$ such that for every $x,y \in M$, $xy \in P$ if and only if either $x \in P$ or $y \in P$. For any given monoid $M$, let $\Spec(M)$ denote the {\em prime spectrum of $M$}, that is, the set of all prime ideals of $M$. It is clear that the union of prime ideals is a prime ideal, and that the empty set and the set of all non-invertible elements of $M$ are prime ideals, which are respectively the least and the greatest prime ideals of $M$. Thus $\Spec(M)$ is a semigroup (commutative and with an identity) with respect to union. More precisely, $\Spec(M)$ is a topological semigroup, and $\Spec$ is a functor from the category of commutative monoids to the category of spectral topological spaces \cite{Pirashvili}. For any commutative monoid $M$, let $\asymp$ be the least congruence on $M$ for which every element in the quotient monoid $M/\!\!\asymp$ is idempotent. The congruence classes of $M$ modulo $\asymp$ are the archimedean components of $M$ \cite[Section~1]{Proc}. The canonical projection $M\to M/\!\!\asymp$ induces an isomorphism $\Spec(M)\cong \Spec(M/\!\!\asymp)$ of topological semigroups \cite[Lemma 2.3]{Pirashvili}. 

\begin{Prop}\label{3.1}
The only finitely generated prime ideals of $\NX$ are $\emptyset$, $(x)$ and the ideals of the form $(p_1,\dots,p_n)$ or $(p_1,\dots,p_n, x)$, where $n \in \N$ and $p_1,\dots,p_n$ are prime numbers.
\end{Prop}

\begin{proof}
Let $I\subset \NX$ be a non-empty finitely generated prime ideal of $\NX$ generated by $g_1(x),\dots, g_m(x)$, that is, $I=g_1(x)\NX\cup g_2(x)\NX\cup\dots\cup g_m(x) \NX$. We can assume, without loss of generality, that $m$ is the minimal number of generators of $I$. Notice that, under this hypothesis on $m$, all the generators are atoms of $\NX$. Indeed, if $g_1(x)$ is not an atom of $\NX$ and $g_1(x)=f_1(x)\cdot \dots \cdot f_k(x)$ is a factorization of $g_1(x)$ into atoms of $\NX$, then $f_j(x) \in I$ for some $j=1,\dots, k$, since $I$ is a prime ideal. This means that there exists $i=2,\dots, m$ such that $f_j(x) \in g_i(x)\NX$, that is, there exists $i=2,\dots, m$ such that $g_i(x)\mid_\N f_j(x)$. Since $f_j(x)$ is an atom, it follows that $g_i(x)=f_j(x)$ and so, in particular, $g_1(x)\NX\subseteq g_i(x)\NX$. Therefore $I=g_1(x)\NX\cup g_2(x)\NX\cup\dots\cup g_m(x) \NX=g_2(x)\NX\cup\dots\cup g_m(x) \NX$ is generated by $g_2(x),\dots, g_m(x)$, which contradicts the minimality of $m$.

Now let $\Omega$ be the subset of $\{1,\dots, m\}$ consisting of all the indices $i$ for which $g_i(x)$ is not a prime element of $\NX$. We want to show that $\Omega=\emptyset$. Assume the contrary. Up to relabelling the generators of $I$, we can assume that $\Omega=\{1,2,\dots,t\}$ for some $1\leq t\leq m$. Fix an irreducible polynomial $h(x)\neq x$ of degree 1 in $\NX$ such that $g_i\nmid_\N h(x)$ for every $i=1,\dots,m$, that is, $h(x)\notin I$. By Lemma \ref{lemma3}, there exists $z(x) \in \Z[x]\setminus \NX$ of degree $\deg(z)\geq 2\max_i\{\deg(g_i)\}$, irreducible in $\Z[x]$ such that $h(x)z(x) \in \NX$ and $g_i(x)z(x) \in \NX$ for every $i\in \Omega$. Notice that, since $h(x)z(x) \in \NX$, the leading coefficient of $z(x)$ must be positive. We have that, for every $i \in \Omega$,  $g_i(x)z(x)h(x) \in g_i(x)\NX \subseteq I$ but neither $g_i(x)z(x)$ nor $h(x)$ belongs to $g_i(x)\NX$. Therefore, by Remark \ref{atomprime}, there exists an index $j:=j(i)\neq i$ in $\Omega$ with $g_i(x)z(x) \in g_j(x)\NX$, because $h(x) \notin I$ (in particular, this shows that if $\Omega\neq \emptyset$, then $\Omega$ has at least two elements). Notice that $g_i(x)z(x)\neq g_j(x)$ because $\deg(g_i z)>\deg(g_j)$. Up to reordering the indices, we see that there exist an element $\ell\neq 1$ in $\Omega$ and polynomials $f_1(x),\dots,f_\ell(x) \in \NX$ of degree at least 1, such that
$$
g_2(x)z(x)=g_1(x)f_1(x),\  g_3(x)z(x)=g_2(x)f_2(x),\  \dots,\  g_1(x)z(x)=g_\ell(x)f_\ell(x).
$$
Using these relations, we get $g_2(x)z(x)^\ell=g_2(x)f_1(x)f_2(x)\cdot \dots \cdot f_\ell(x)$, and so $z(x)^\ell=f_1(x)f_2(x)\cdot \dots \cdot f_\ell(x)$. Since $z(x)$ is irreducible in $\Z[x]$ and its leading coefficient is positive, we must have $z(x)=f_1(x)=\dots=f_\ell(x)\in \NX$, a contradiction. 
\end{proof}

For an example of a monoid with a finitely generated prime ideal that is not generated by prime elements, consider the additive monoid $\N_0\setminus\{1\}$. In this monoid, the atoms are only $2$ and $3$, and there are no prime elements. The ideal generated by $2$ and $3$ is the ideal of all non-units, which is the greatest prime ideal of the monoid.

\bigskip

As we saw in Proposition~\ref{Fede},  the only prime elements of $\NX$ are the positive prime integers and the polynomial $x$. Since every polynomial in $\NX$ is the product of its content and its primitive part in a unique way, we have a direct-product decomposition $\NX=\N\times P$, where $P$ denotes the multiplicative monoid of all primitive polynomials (polynomials with content 1) belonging to $\NX$. Every polynomial in $P$ is the product of a power $x^n$ of the prime element $x$ and a primitive polynomial in $\NX$ with non-zero constant term in a unique way, so we have a direct-product decomposition $P=\N \times P_0 $, where $P_0 $ denotes the multiplicative monoid of all primitive polynomials with non-zero constant term. Thus $\NX$ is the direct product of a free commutative monoid $\N_0^{(I)}$ and the monoid $P_0 $. Here the free set $I$ of generators of $\N_0^{(I)}$ consists of all positive prime integers and the polynomial $x$, which are the prime elements in the monoid $\NX$, as we have said above. Hence all problems concerning the non-uniqueness of factorization in $\NX$ occur in its submonoid $P_0 $.

\begin{Prop} There is an automorphism $\varphi\colon P_0\to P_0$ that maps every $a_nx^n+_{n-1}x^{n-1}+\dots+a_1x+a_0\in P_0$ of degree $n$ to $a_0x^n+a_1x^{n-1}+\dots+a_{n-1}x+a_n$. Moreover, $\varphi$ is an involution in $P_0$, i.e., $\varphi^2$ is the identity automorphism of $P_0$.\end{Prop}

\begin{proof} Notice that $\varphi(f(x))=f(\dfrac{1}{x})x^{\delta(f(x))}$ for every $f(x)\in P_0$ of degree $\delta(f(x))$. Now the verification is straightforward.\end{proof}

It would be interesting to determine all the prime ideals of the monoids $\NX$ or $P_0$, not only the finitely generated ones determined in Proposition~\ref{3.1}. It would be nice to determine at least the prime ideals of height $1$ of these two monoids. In fact, we will see in Section~\ref{n} that the monoid $\NX$ is not a Krull monoid, but can be described using techniques similar to those employed in the study of the structure of Krull monoids, essentially making use of derivations instead of monoid morphisms. Thus the description of prime ideals of height $1$ in $\NX$ could be relevant.

Here are some prime ideals of the monoid $P_0$. There are two canonical homomorphisms of $P_0$ into the multiplicative monoid $\N$. The first is $v_0\colon P_0\to\N$ (valuation in $0$), defined by $v_0(f(x))=f(0)$ for every $f(x)\in P_0$. The second is $\ell\colon P_0\to\N$, which associates to every $f(x)\in P_0$ its leading coefficient $\ell(f(x))$. Since the multiplicative monoid $\N$ is a free commutative monoid with free set of generators the set $\Cal P_0$ of all positive prime integers, i.e., $\N\cong\N_0^{(\Cal P_0)}$, it follows that there is, for each positive prime integer $p$, a monoid morphism $v_{0,p}\colon P_0\to\N_0$ into the additive monoid $\N_0$, defined,  for every $f(x)\in P_0$, by $v_{0,p}(f(x))=\,$``exponent of $p$ in a prime factorization of $f(0)$''. There is also, for each prime integer $p$, a monoid morphism $\ell_{p}\colon P_0\to\N_0$, defined,  for every $f(x)\in P_0$, by $\ell_{p}(f(x))=\,$``exponent of $p$ in a prime factorization of the leading coefficient of $f(x)$''. Thus we have two countable families of valuations $P_0\to\N_0$, indexed in the prime numbers. A further valuation $\delta\colon P_0\to\N_0$ is given by $\delta(f(x))=\,$``degree of the polynomial $f(x)$''.

\medskip

Correspondingly, we have prime ideals of the monoid $P_0$:

(1) The prime ideal of $P_0$ consisting of all polynomials $f(x)\in P_0$ with $f(0)\ne 1$.

(2) The prime ideal of $P_0$ consisting of all polynomials $f(x)\in P_0$ with leading coefficient $\ne 1$.

(3) For every prime integer $p$, the prime ideal of $P_0$ consisting of all polynomials $f(x)\in P_0$ such that $p|f(0)$.

(4) For every prime integer $p$, the prime ideal of $P_0$ consisting of all polynomials $f(x)\in P_0$ with leading coefficient divisible by $p$.

(5) The greatest prime ideal $P_0\setminus\{1\}$ of the monoid $P_0$.

\bigskip

In the following, we want to investigate some properties of the polynomials $f(x) \in \Z[x]$ for which $f(x)\Z[x]\cap \NX \neq \emptyset$. Consider
$$
E:=\{f(x) \in \Z[x]\mid f(x)\Z[x]\cap \NX \neq \emptyset\}.
$$
Then $E$ is a submonoid of $\Z[x]$ containing $\NX$. It is immediate to check that both the containments $\NX\subseteq E$ and $E\subseteq \Z[x]$ are proper. For instance, $x-1 \in \Z[x]\setminus E$ while $x^2-x+1 \in E \setminus \NX$. The atoms of $E$ are exactly the polynomials of $E$ that are irreducible in $\Z[x]$, that is, $\Cal A(E)=\Cal A(\Z[x])\cap E$. This follows from the fact that for any pair of polynomials $f(x) \in E$ and $g(x)\in \Z[x]$, if $g(x)\mid_\Z f(x)$, then $g(x) \in E$. In particular, every element of $E$ can be written in a unique way, up to associates, as a product of atoms.

For every element $\lambda(x)\in \Cal A(E)$ we can define the ideal 
$$
P_\lambda:=\lambda(x)\Z[x]\cap \NX=\lambda(x)E\cap \NX,
$$
which turns out to be a prime ideal of $\NX$. It corresponds to the valuation $v_\lambda: \NX \rightarrow (\N_0,+)$ that assigns to any polynomial $f(x) \in \NX$ the exponent of $\lambda(x)$ in a factorization of $f(x)$ into irreducible elements of $\Z[x]$. Notice that $P_\lambda$ is the contraction of the principal prime ideal $\lambda(x)E$  of $E$. However, the ideal $P_\lambda$ is principal if and only if $\lambda(x)$ is a prime element of $\NX$. In this case, $P_\lambda=\lambda(x)\NX$ is the principal ideal generated by $\lambda(x)$. On the other hand, if $\lambda(x)$ is in $\NX$ but it is not a prime element or if $\lambda(x) \in \Cal A(E)\setminus \NX$, then $P_\lambda$ is not finitely generated. A set of generators for $P_\lambda$ is given by $\lambda(x)\Z[x]\cap \Cal A (\NX)$.

\begin{Prop}
For every $\lambda(x) \in \Cal B:=\Cal A(E)\cap \NX$, define
$$
A_\lambda:=\{f(x)\in \Z[x] \mid \lambda(x)^k f(x) \in \NX\ \mbox{\rm for some } k \in \N_0\}.
$$
Then $A_\lambda$ is a semiring containing $\NX$ and $\NX=\cap_{\lambda \in \Cal B}A_\lambda$. 
\end{Prop}
\begin{proof}
First of all, notice that $x^n+2 \in \Cal B$ for every $n \in \N$. Now, if $f(x) \in \Z[x]\setminus \NX$ is a polynomial of degree $d$, then $(x^{d+1}+2)^kf(x) \notin \NX$ for every $k \in \N$. Therefore $f(x) \notin\cap_{\lambda \in \Cal B}A_\lambda$.
\end{proof}

\section{Chains of principal ideals}\label{cpi}

As we have already remarked, there are two partial orders on the monoid $\NX$: the partial orders $|_{\Z}$ and $|_{\N}$, which is the algebraic order on $\NX$.
Clearly, for every $f,g\in\NX$, $f|_{\Z}g$ if and only if the principal ideal $(f)$ of $\Z[x]$ generated by $f$ contains the principal ideal $(g)$ of $\Z[x]$ generated by $g$. Thus the mapping $f\mapsto (f)$ is an embedding of partially ordered sets of $(\NX,|_{\Z})$ into the partially ordered set of all principal ideals of $\Z[x]$, which is a partially ordered subset of the lattice of all ideals of $\Z[x]$. Notice that both the partially ordered set of all principal ideals of $\Z[x]$ and the partially ordered set of all ideals of $\Z[x]$ are modular lattices, that the first is a partially ordered subset of the second, but is not a sublattice of the second (in the first, the least upper bound is given by the least common divisor, and in the second the least upper bound is given by the sum of ideals, but in $\Z[x]$ there are ideals that are not principal). 

As we have already said, for a commutative monoid $M$, an {\em ideal} in $M$ is a subset $I$ of $M$ such that $IM\subseteq I$. A subset
$I$
of a commutative semiring
$S$
is an {\em ideal}
of $S$ if
$a+b\in I$ for all $a,b\in I$ and $IS\subseteq I$. An ideal is {\em principal} if it is of the form $aM$ for some $a\in M$ (of the form $aS$ for some $a\in S$). Here $M$ is a commutative monoid and $S$ a commutative semiring. 

\medskip

The monoid $\NX$ does not satisfy the ascending chain condition on ideals. Indeed, we can recursively define the following ideals of $\NX$: set $I_1:=(x+1)\NX$ and, for $n\geq 2$, set $I_n:=I_{n-1}\cup (x+n)\NX$. It is clear that the chain $I_1\subsetneq I_2\subsetneq \dots$ is not stationary.

\medskip

The set $\Cal L_p(\NX)$ of all principal ideals of $\NX$ is partially ordered by set inclusion, and has the empty ideal as its least element and the improper ideal $\NX$ as its greatest element. Every principal ideal in $\NX$ has a unique generator. There is an anti-isomorphism of partially ordered sets, that is, an order-reversing one-to-one correspondence, between the partially ordered set $(\NX,|_{\N})$ and the partially ordered set $(\Cal L_p(\NX),\subseteq)$.

Given any factorization $a=a_1a_2\dots a_n$ in $\NX$, we can associate to the factorization the chain of principal ideals $$a\NX\subseteq a_1a_2\dots a_{n-1}\NX\subseteq a_1a_2\dots a_{n-2}\NX\subseteq a_1\NX\subseteq \NX$$
in the spirit of
\cite[Theorem~2.1]{FacchiniFassina}. More precisely, let $a$ be an element of $\NX$, $$\Cal F(a):=\{\,(a_1,a_2,\dots,a_n)\mid n\ge 1,\ a_i\in \NX,\ a_1a_2\dots a_n=a\,\}$$ be the set of all factorizations of $a$, and \begin{equation*}
\begin{split}\Cal C_a:=\{\,(I_0=a\NX\subseteq I_1\subseteq I_2\subseteq\dots\subseteq I_{n-1}\subseteq I_n= \NX) \mid \\ n\ge 1, I_j\ \mbox{\rm a principal ideal of }\NX\,\}\end{split}
\end{equation*} be the set of all finite chains of principal ideals from $a\NX$ to $\NX$.
Let $f\colon \Cal F(a)\to \Cal C_a$ be the mapping defined by sending any factororization $$(a_1,a_2,\dots, a_n)\in \Cal F(a)$$ to the chain $$(a\NX=a_1a_2\dots a_{n}\NX\subseteq a_1a_2 \dots  a_{n-1}\NX \subseteq  \dots\subseteq  a_1\NX \subseteq  \NX ).$$ Then the mapping $f$ is a one-to-one correspondence. Hence the factorizations of $a$ in $\NX$ are decribed by the finite chains of principal ideals in the interval $[a\NX,\NX]_{\Cal L_p(\NX)}$. Notice that, for every $a\in\NX$, the interval $[a\NX,\NX]_{\Cal L_p(\NX)}$ is always a finite set. It would be very interesting to determine which finite partially ordered sets can be realized as $[a\NX,\NX]_{\Cal L_p(\NX)}$ for some $a\in\NX$.

\section{Krull monoids, derivations}\label{n}

An integral domain $R$ with field of fractions $K$ is a {\em Krull domain} if there exists a set $\{\,v_i\mid i\in I\,\}$ of valuations $v_i\colon K\to {\Z}\cup\{+\infty\}$ with the following two properties: (1) $R=\{\,x\in K\mid v_i(x)\ge0$ for every $i\in I\}$;  (2)~for every $x\in K$ the set $\{\,i\in I\mid v_i(x)\ne0\,\}$ is finite.  The notion of Krull domain has been generalized to the case of commutative monoids in \cite{Chouinard} by Chouinard, who introduced the notion of {\em Krull monoid}. As good references for the theory of Krull monoids, we mention the monographs \cite{GHK06} and \cite{HK98}. Here, we briefly recall the definition and some noteworthy results.

A monoid homomorphism $f\colon M\to M'$ is called a {\em divisor homomorphism\/} if, for
every
$x,y\in M$, $f(x)\le f(y)$ implies $x\le y$. Here $\le$ denotes the algebraic pre-orders. A monoid $M$ is a {\em
Krull monoid} if there exists a divisor homomorphism of $M$ into a factorial monoid. A commutative integral domain $R$ is a Krull domain if and only if the monoid $R^*:=R\setminus\{0\}$ is a Krull monoid \cite{UKrause}. Beyond integral domains, a prime polynomial identity ring is a Krull ring if and only if its monoid of regular elements is a Krull monoid \cite{Wau84}.

\begin{Prop}\label{Chouinard} Let $M$ be a multiplicative cancellative commutative monoid with quotient group $G(M)$. The following conditions are equivalent:
\begin{enumerate}
\item[{\rm (a)}] $M$ is a Krull monoid.
\item[{\rm (b)}]
There exists a set $\{\,v_i\mid i\in I\,\}$ of non-zero group morphisms $$v_i\colon G(M)\to {\Z}$$ such that: {\rm (1)} $M=\{\,x\in G(M)\mid v_i(x)\ge0$ for every $i\in I\,\}$;  and {\rm (2)} for every $x\in G(M)$, the set $\{\,i\in I\mid v_i(x)\ne0\,\}$ is finite.
\item[{\rm (c)}] There exist an abelian group $G$, a set $I$ and a subgroup $H$ of the free abelian group ${\Z}^{(I)}$ such that $M\cong G\oplus
(H\cap\N_0^{(I)})$.\end{enumerate}\end{Prop}

A monoid $M$ is {\em root closed}, or {\em integrally closed},  if for every element $a$ in the quotient group $G(M)$ of $M$,  $a^n \in M$ for some integer $n\geq 1$ implies $a \in M$ \cite{Chouinard}. The monoid 
$M$ is {\em completely integrally closed} if  $a \in M$, $x \in G(M)$ and $ax^n \in M$ for every $n \geq 0$ implies $x \in M$. Of course, if $M$ is completely integrally closed, then $M$ is integrally closed, and the two conditions are equivalent if $M$ satisfies the a.c.c. on ideals, so, in particular, if $M$ is finitely generated \cite[p.~1460]{Chouinard}.

It is easily seen that every cancellative Krull monoid is root closed. The following example, due to Florian Kainrath, to whom we are grateful, shows that the cancellative monoid $\NX$ is not root closed, hence is not a Krull monoid. 

\begin{example}\label{Kainrath}
Consider the polynomial $f(x): = x^4 + 2x^3 - x^2 + 4x + 2$. It is immediate to see that $f(x) (x+1) = x^5 + 3x^4 + x^3 + 3x^2 + 6x + 2$ and $f(x)^2 = x^8 + 4 x^7 + 2x^6 + 4x^5 + 21 x^4 + 12 x^2 + 16x + 4$. This means that $f(x) \in G(\NX) \setminus \NX$ and $f(x)^2 \in \NX$, so $\NX$ is not a root closed monoid.
\end{example}

Recall that a commutative, cancellative monoid $M$ with torsion-free
quotient group $G(M)$ is a Krull monoid if and only if $M$ is completely integrally closed and satisfies the ascending chain condition on divisorial ideals in $M$. We have shown that our monoid $\NX$ has free quotient group, is not a Krull monoid and is not completely integrally closed because it is
not integrally closed. The monoid $\NX$ does not satisfies the a.~c.~c.~
on divisorial ideals, as the strictly ascending chain $\NX\subset
x^{-1}\NX\subset x^{-2}\NX\subset \dots$ shows.

Beyond Krull monoids, the best investigated objects in factorization theory are C-monoids (see \cite[Definition~2.9.5]{GHK06}). Our monoid $\NX$ is not a C-monoid because its set of distances is infinite (\cite[Theorem~2.3]{CCMS09}) and C-monoids do a have a finite set of distances \cite[Theorem~3.3.4]{GHK06} (note that the finiteness of the catenary degree implies the finiteness of the set of distances).

\begin{remark}
Let $R$ be the {\em root closure} of $\NX$, that is the monoid
$$
R:=\{q(x) \in G(\NX)\mid q(x)^n \in \NX \mbox{ for some integer } n\geq 1 \}.
$$
As Example \ref{Kainrath} shows, $R$ is a submonoid of $G(\NX)$ strictly containing $\NX$. Moreover, $R \subsetneq \Z[x]^*$. To prove it, let $q(x)$ be an element of $R$. We can write $q(x)=f(x)/g(x)$ for suitable polynomials $f(x),g(x) \in \NX$. Since $q(x) \in R$, it follows that there exists an integer $n \geq 1$ such that $g(x)^n \mid_\N f(x)^n$. In particular, $g(x)$ must divide $f(x)$ in $\Z[x]$, that is $q(x)$ is in $\Z[x]$. Finally, it is immediate to see that $x-1 \notin R$.

There is an isomorphism of topological semigroups $\Spec(R)\rightarrow \Spec(\NX)$ given by $P\mapsto P\cap \NX$ \cite[Corollary~2.4]{Pirashvili}.
\end{remark}

\begin{remark}
Let us show that there are no transfer homomorphisms from the monoid $\NX$ into a block monoid. We refer to  \cite[Definitions~2.5.5 and 3.2.1]{CCMS09} for the definition of block monoid and transfer homomorphism, respectively. It can be shown, using \cite[Lemma~6.4.4]{CCMS09} and \cite[Proposition~3.2.3]{CCMS09}, that if $M$ is a monoid admitting a transfer homomorphism into a block monoid, then the lengths of factorizations of elements of the form $ab$ with $a,b\in \Cal A(M)$ and $a$ fixed, are bounded by a constant depending only on $a$.
Now, consider the polynomial $x+1 \in \Cal A(\NX)$. We have seen in Section~\ref{Sec2} (see also \cite[Theorem~2.3]{CCMS09}) that $(x+n)^n(x^2-x+1)$ is an atom of $\NX$ for every $n \in \N$ and that $g_{1,n}(x)=[x+1]\cdot[(x+n)^n(x^2-x+1)]$ admits a factorization of length $n+1$, namely $g_{1,n}(x)=[x+n]^n\cdot[(x^2-x+1)(x+1)]$. It follows that the lengths of factorizations of elements of the form $(x+1)f(x)$, with $f(x)\in \Cal A(M)$, are not bounded and so there are no transfer homomorphisms from the monoid $\NX$ into a block monoid.
\end{remark}

We now follow \cite{Bou} and \cite[Section~2]{Meltem}. Let $K$ be a commutative ring with identity. A (not-necessarily associative) {\em $K$-algebra} is any unitary $K$-module $M$ with a $K$-bilinear mapping $(a,b)\mapsto ab$ of $M\times M$ into $M$ (equivalently, a $K$-linear mapping $M\otimes_KM\to M$). If $M$ is any (not-necessarily associative) $K$-algebra and $a\in M$, multiplication by $a$ is a mapping $\lambda_a\colon M\to M$, defined by $\lambda_a(b)=ab$ for every $b\in M$, which is an element of the associative $K$-algebra $\End_K(M)$ of all endomorphisms of the $K$-module $M$.

 As multiplications provide endomorphisms of the module, commutators give derivations. A {\em derivation} of a $K$-algebra $M$ is any $K$-linear mapping $D\colon M\to M$ such that $D(ab)=(D(a))b+a(D(b))$ for every $a,b\in M$. If $M$ is any $K$-algebra and $D,D'$ are two derivations of $M$, then the composite mapping $DD'$ is not a derivation of $M$ in general, but the commutator $DD'-D'D$ is. For any Lie $K$-algebra $M$ and any element $a\in M$, the mapping $\lambda_a\colon M\to M$, defined by $\lambda_a=[a,-]$, is an element of the Lie $K$-algebra $\Der_K(M)$ of derivations of the $K$-algebra $M$. 

\begin{Prop}\label{m} Let $d=\frac{d}{dx}\colon\Z[x]\to \Z[x]$ be the usual derivation of polynomials and $v_0\colon\Z[x]\to \Z$ the ring morphism $v_0\colon f\mapsto f(0)$. Then:  {\rm (1)} $\N_0[x]=\{\,a\in \Z[x]\mid v_0d^n(a)\ge0$ for every $n\ge0\,\}$;  and {\rm (2)} for every $a\in \Z[x]$, the set $\{\,n\mid n\ge0,\ v_0d^n(a)\ne0\,\}$ is finite.\end{Prop}

Here $d \colon \Z[x]\to\Z[x]$ is the derivation on  the commutative ring $\Z[x]$ of polynomials in one indeterminate $x$, defined by $d (\sum_{i=0}^na_ix^i)=\sum_{i=1}^nia_ix^{i-1}$, and, for  the composition $d ^n$ of $d  \ n$ times, we have that $d ^n(ab)=\sum_{i=0}^n\binom{n}{i}d^{n-i}(a)d^i(b)$ by the Leibniz Rule.

\begin{Prop}\label{m'} Let $d\colon\Z[x]\to \Z[x]$ be the usual derivation of polynomials and $v_0\colon\Z[x]\to \Z$ the ring morphism $v_0\colon f\mapsto f(0)$. Consider the mapping $\Delta:=\prod_{n\ge0}v_0d^n\colon \Z[x]\to \Z^{(\N_0)}$ Then $\N_0[x]=\Delta^{-1}(\N_0^{(\N_0)})$.\end{Prop}

 Notice that $d^n$ is exactly the $n$-derivative ${d^n}/{dx^n}$, so that $v_0d^n$ is the $n$-the derivative at the point $x=0$. 
 
 \medskip
 
Now it is natural to introduce the formal differential operator ring over $\Z[x]$. The reason of this is due to the statements of Propositions~\ref{m} and~\ref{m'}, where we consider the derivation $d $ on the ring $\Z[x]$. Both $d $ and multiplications by elements of $\Z[x]$ can be viewed as group morphisms $\Z[x]\to \Z[x]$. Here, for every element $g$ of $\Z[x]$, multiplication by $g$ is such that $f\mapsto gf$. In other words, $d \in\End_\Z(\Z[x])$ and $\Z[x]\subseteq \End_\Z(\Z[x])$, where $\End_\Z(\Z[x])$ denotes the endomorphism ring of the free abelian group $\Z[x]$.
Thus, for the elements  $d $ and $g\in\Z[x]$, we have that $d  g=gd +d (g)$ in $\End_\Z(\Z[x])$, because,
for every $f\in \Z[x]$, $d (gf)=d (g)f+gd (f)$, so that the endomorphisms $d  g$ and $gd +d (g)$ of the abelian group $\Z[x]$ coincide. 

Thus the elements of the subring of $\End_\Z(\Z[x])$ generated by $\Z[x]$ and  $d $ can be written in the form $\sum_{i=0}^nf_i(x)d ^i$, where $d  f(x)=f(x)d +d (f(x))$. Equivalently, it is the subring of $\End_\Z(\Z[x])$ generated by $x$ and  $z:=d $, where $zx=xz+1$. Thus we set $A_1(\Z):=\Z\langle X,Z\rangle/(ZX-XZ=1)$, where $\Z\langle X,Z\rangle$ is the free ring on two objects $X$ and $Z$ and $(ZX-XZ=1)$ is the principal two-sided ideal of $\Z\langle X,Z\rangle$ generated by $ZX-XZ-1$.

We now follow \cite[Section~2]{FacNazeErice}. If  $d $ is a derivation on a ring $A$, let $B: = A [z; d ]$ denote the ring of differential operators. Thus the ring $B$ is the ring of all linear differential operators
$$
a_n(x) \frac{d^n}{dx^n}+a_{n-1}(x)
\frac{d^{n-1}}{dx^{n-1}}+\dots+a_{1}(x) \frac{d}{dx}+
a_0(x),
$$
with $a_n(x), a_{n-1}(x),\dots,a_0(x)\in\Z[x]$. Then $A$ is a left $B$-module defining  $z^n a = d ^n (a)$ for every $a \in A$. If $I $ is a subbimodule of $_BA_A$, then $d $ is a derivation on 
$A/I$. 

We will denote by  $A^{d }$ the subring of $A$ consisting of all elements $a\in A$ with $d (a)=0$. Taking as $d $ the usual derivation on the ring $A:=\Z[x]$, we have that $A$ is a left module over the ring $B: = \Z[x] [z; d ]$. Here $B$ is the $\Z$-algebra in two non-commuting indeterminates $x$ and $z$ subject to the relations $z^n x^m = d ^n (x^m)=0$ if $n>m$, and $z^n x^m = d ^n (x^m)=m(m-1)\dots(m-n+1)x^{m-n}=\frac{m!}{(m-n)!}x^{m-n}$ if $n\le m$.

It is easy to see that the endomorphism ring $\End(_B A)$ of $A$ as a left $B$-module is isomorphic to $A^d $. Indeed, an element in $\End(_B A)$ is an $A$-module morphism $\varphi:A\rightarrow A$ such that $\varphi(d (a))=d  \varphi (a)$ for every $a \in A$. Since $\varphi$ is given by multiplication on the right by an element $x \in A$, we must have $d (a)x=d (ax)=d (a) x +ad (x)$ for every $a \in A$, and this occurs exactly when $x \in A^d $ .

Recall that if $A$ is an algebra over a field $k$ of characteristic zero, $d$ is a locally nilpotent derivation of $A$ and $d (x) = 1$ for some $x\in A$, then any non-zero element $f\in A$ is of the form $f = \sum_{ i = 0}^n c_i x^i$, where $c_i \in A^{d}$ and $c_n$ is non-zero \cite[Lemma~3.9]{BorgesLomp}. This representation is unique, because if some linear combination $\sum_{ i = 0}^n c_i x^i$ is zero, then, by applying $n$ times $d $, we see that $n! c_n=0$. As $A$ is an algebra over a field of characteristic zero, $c_n=0$. Inductively, all the coefficients must be zero. In particular $A=A^d [x;d]$, where $d$ is the derivation $d(a)=xa-ax$ for every $a\in A^d $. Note that, for any $a\in A^d $ and $n\geq 0$, we have  $x^n a = \sum_{k=0}^n \binom{n}{k} d^{k}(a)x^{n-k}$.

 Now that we have that $\Z[x]$ is a left $ \Z[x] [z; d ]$-module, we have left multiplication by the elements $z^n\in \Z[x] [z; d ]$, which are group morphisms $\lambda_{z^n}\colon \Z[x]\to\Z[x]$. The mapping $v_0\colon\Z[x]\to\Z$ is also a group morphism, and $\N_0[x]$ consists exactly of the elements $f\in\Z[x]$ such that $v_0(z^nf)\ge0$, similarly to the case of Krull monoids.

\section{An algorithm to determine factorizations in $\NX$}

The multiplicative monoid $\NX$ can be embedded into the monoid $\N_0^{\N_0}$, direct product of countably many copies of the multiplicative monoid $\N_0$. It suffices to consider the monoid embedding $\varepsilon\colon\NX\to \N_0^{\N_0}$, defined by $\varepsilon\colon f(x)\mapsto (f(n))_{n\in\N_0}$. The mapping $\varepsilon$ is injective, because if $f(x),g(x)\in\NX$ and $f(n)=g(n)$ for every $n\in\N_0$, then $f(x)-g(z)$ is a polynomial with coefficients in $\Z$ with infinitely many distinct zeros, hence is the zero polynomial. In this section, we will consider an embedding of the monoid $\NX$ into $\N^2$, and we will use it to describe an algorithm to factorize polynomials in $\NX$.

Let $f$ be a polynomial in $\NX$ and let $\alpha(f)$ denote the greatest of the coefficients of $f$. Then $f(x)$ is completely determined by any pair of natural numbers $(a,b)$, where $a>\alpha(f)$ and $b=f(a)$. Indeed, the coefficients of $f$ correspond to the digits of the representation of $f(a)$ in the base-$a$ numeral system. To be more precise, given any pair $(a,b)\in \N_0^2$ such that $a>\alpha(f)$ and $b=f(a)$, we can divide $b$ by $a$, getting $b=aq_0+r_0$ for some natural numbers $q_0$ and $0\leq r_0< a$. Since $a>\alpha(f)$, $r_0$ coincides with the constant term of $f$. Now we can divide $q_0$ by $a$, obtaining $q_0=q_1a+r_1$. Iterating this process until $q_n$ is zero for some $n$, we get a series of reminders $r_0, r_1,\dots, r_n$. Of course, $r_nr_{n-1}\dots r_1r_0$ represents $b$ in the base-$a$ numeral system and it is easy to check that $f(x)=r_nx^n+r_{n-1}x^{n-1}+\dots+r_1x+r_0$.
Also notice that if $f(x)$ is determined by the pair $(a,b)$, then the degree of $f$ is equal to the number $$\eta(a,b):=\max\{\,n\in \N \mid a^n<b\,\}.$$

Using this fact, we have an injective map $\NX \hookrightarrow \N_0^2$ which sends any polynomial $f(x)\in\NX$ to the pair $(\alpha(f)+1,f(\alpha(f)+1))$. Moreover, we have a partition of $\N^2$ corresponding to the equivalence relation given by $(a_1,b_1)\sim(a_2,b_2)$ if there exists a polynomial $f(x)\in \NX$ such that $\alpha(f)<a_1,a_2$ and $f(a_i)=b_i$ for $i=1,2$. Thus any pair of the form $(a,b)$ with $a>b$ corresponds to the constant polynomial $f(x)=b$, while pairs on the diagonal, that is, pairs of the form $(a,a)$, all correspond to the polynomial $f(x)=x$. Finally, pairs of the form $(1,b)$ do not identify any element of $\NX$.

Notice that we can embed $\NX$ into $\N_0^2$ also sending any polynomial $f(x) \in \NX$ to the pair $(f(1)+1,f(f(1)+1))$, since $f(1)+1>\alpha(f)$. This map has the advantage of depending only on the evaluations of the polynomial and not (at least in a direct way) on the coefficients of $f$.

Given any (primitive) polynomial $f(x)$ in $\NX$, we can factor it with the following algorithm:
\begin{enumerate}
\item
Set $a:=\alpha(f)+1$ and $b:=f(a)$.
\item
Consider all possible factorizations of $b=b_1\dots b_k$ such that $b_j>1$ for all $j=1,\dots, k$ and $\deg(f)=\eta(a,b)=\sum_{j=1}^k \eta(a,b_j)$.
\item
For any obtained factorization $b=b_1\dots b_k$, consider the pairs $(a,b_1),\dots,$ $(a,b_k)$.
\item
The pairs $(a,b_j)$ of the previous step identify polynomials $f_{(a,b_j)}(x)$ in $\NX$ (write $b_j$ in base $a$).
\item
For any possible factorization of $b$ in $(2)$, we get a polynomial $g(x)=\prod_{i=1}^k f_{(a,b_j)}(x)$.
If $\alpha(g)\leq a$, then $g(x)=f(x)$.
\end{enumerate}

\begin{example}
Consider the polynomial $f(x)=x^5+x^4+x^3+x^2+x+1$ of the Introduction. Following the previous algorithm, we have:
\begin{enumerate}
\item
$a=2$ and $b=f(2)=63$.
\item
We can write $63=3\cdot21=9\cdot 7$. We must discard the factorization $63=3\cdot 3 \cdot 7$, because it does not satisfy the condition $\eta(a,b)=\sum_{j=1}^k \eta(a,b_j)$.
\item
We get the two pairs $(2,3)$ and $(2,21)$ from the factorization $63=3\cdot21$ and the two pairs $(2,9)$ and $(2,7)$ from the factorization $63=9\cdot 7$.
\item
We find the polynomials $f_{(2,3)}(x)=x+1$, $f_{(2,21)}(x)=x^4+x^2+1$, $f_{(2,9)}(x)=x^3+1$ and $f_{(2,7)}(x)=x^2+x+1$. 
\item
We get two factorizations:
$$
f(x)=(x+1)(x^4+x^2+1)=(x^3+1)(x^2+x+1).
$$
\end{enumerate}

\end{example}

\begin{example}
Consider the polynomial $f(x)=x^2+10x+3$. We have:
\begin{enumerate}
\item
$a=11$ and $b=f(11)=234$.
\item
We can write $234=13\cdot 18$. We must discard all the other factorizations, namely $234=2\cdot 9 \cdot 13=2\cdot 117=9\cdot 26$, because they do not satisfy the condition $\eta(a,b)=\sum_{j=1}^k \eta(a,b_j)$.
\item
We have the two pairs $(11,13)$ and $(11,18)$ from the factorization $234=13\cdot 18$.
\item
We get the polynomials $f_{(11,13)}(x)=x+2$ and $f_{(11,18)}(x)=x+7$. 
\item
But in this case $\alpha( f_{(11,13)} f_{(11,18)})=14>11$, so
$$
f(x)\neq f_{(11,13)}(x)f_{(11,18)}(x)=x^2+9x+14,
$$
which means that $f(x)$ is irreducible in $\NX$.
\end{enumerate}

\end{example}

\begin{remark}
We can consider the submonoid $N$ of $\NX$ consisting of all polynomials $f(x) \in \NX$ such that $x$ does not divide $f(x)$. Of course, the problem of studying factorizations does not change, but the assignment of the embedding $N\hookrightarrow \N_0^2$ becomes more elegant (and hopefully, more useful): indeed, we can send any polynomial $f(x)$ into the pair $(f(1),f(f(1)))$. In fact, if $f(x)$ is not a power $x^n$ of $x$, then $f(1)>\alpha(f)$ and all we have seen in this section also works taking $f(1)$ instead of $\alpha(f)+1$.
\end{remark}

\bibliographystyle{amsalpha}

\end{document}